\documentclass[12pt]{amsart}

\usepackage{amsmath}
\usepackage{amsfonts}
\usepackage{amsthm}

\usepackage[cmtip,arrow]{xy}
\usepackage{pb-diagram,pb-xy}

\newtheorem{theorem}{Theorem}

\newtheorem{lemma}{Lemma}

\newtheorem{definition}{Definition}

\newtheorem{remark}{Remark}

\begin{document}
\title{Tate-Betti and Tate-Bass numbers}
\author{E. Enochs and S. Estrada and A. Iacob}
\thanks{key words: Tate-Betti and Tate-Bass number, complete projective resolution, eventually periodic complex, Matlis duality}

\begin{abstract}
We define Tate-Betti and Tate-Bass invariants for modules over a commutative noetherian local ring $R$. Then we show the periodicity of these invariants provided that $R$ is a hypersurface. In case $R$ is also Gorenstein, we see that a finitely generated $R$-module $M$ and its Matlis dual have the same Tate-Betti and Tate-Bass numbers.
\end{abstract}

\maketitle

\section{Introduction}

We consider a commutative noetherian local ring $(R, \mathfrak{m}, k)$.\\
It is known that a module $M$ has a complete projective resolution $T \rightarrow P \rightarrow M$ if and only if $M$ has finite Gorenstein projective dimension. We prove that when $M$ is a finitely generated $R$-module of finite Gorenstein projective dimension we can construct a complete projective resolution $T \rightarrow P \rightarrow M$ with both $T$ and $P$ homotopically minimal complexes, and so unique up to isomorphism. Then each of the modules $P_n$ ($n \ge 0$) and $T_n$ ($n \in \mathbb{Z}$) are free modules of finite ranks. The ranks of the modules $P_n$ are usually denoted $\beta_n(M)$ and are called the Betti numbers of $M$. The boundedness of the sequence of Betti numbers of a module $M$, as well as the interplay between the boundedness of the Betti numbers and the eventual periodicity of the module $M$ have been studied intensively (see for example \cite{eisenbud:80:complete.intersections},  \cite{peeva:90:boundedness}, \cite{gulliksen:69:minimal}, \cite{peeva:80:exp.growth.betti}, and \cite{bergh:80:exp.growth.betti}). \\
We focus here on the invariants $\widehat{\beta}_n(M)$, where for each $n \in \mathbb{Z}$, $\widehat{\beta}_n(M)$ is the rank of the module $T_n$. We call these invariants the Tate-Betti numbers of $M$ (see \cite{christensen:14:tate} for another way to define these invariants).\\
For an arbitrary module $N$ we can use an analogous procedure to construct a complete injective resolution $N \rightarrow I \rightarrow U$ where both $I$ and $U$ are homotopically minimal complexes (and hence unique up to isomorphism). Then we can define the Tate-Bass invariants $\widehat{\mu}^n (\mathfrak{p}, N)$ for $n \in \mathbb{Z}$ and $\mathfrak{p} \subset R$ a prime ideal of $R$. \\
Our main results (Theorem 2 and Theorem 3) give sufficient conditions on the residue field $k$ that guarantee the periodicity of the Tate-Betti numbers $\widehat{\beta}_n(M)$ (where $M$ is finitely generated of finite Gorenstein projective dimension) and respectively the periodicity of the Tate-Bass numbers $\widehat{\mu}^n (\mathfrak{m}, N)$. We prove (Theorem 2) that if $k$ has an eventually periodic minimal projective resolution with period $s$, then for every finitely generated $R$-module $M$ we have that the Tate-Betti invariants of $M$ are periodic of period $s$. We also prove (Theorem 3) that under the same hypothesis on $k$ we have that the Tate-Bass invariants $\widehat{\mu}^n (\mathfrak{m}, N)$ are periodic of period $s$, for every module $N$ of finite Gorenstein injective dimension.\\

In the second part of the paper we consider a commutative local Gorenstein ring $(R, \mathfrak{m}, k)$ and a finitely generated $R$-module $M$. We prove that if $T \rightarrow P \rightarrow M$ is a minimal complete projective resolution of $M$, then $M^\nu \rightarrow P^\nu \rightarrow T^\nu$ is a minimal complete injective resolution of $M^\nu$. It follows that for each $n$ we have $\widehat{\mu}_n(\mathfrak{m}, M^\nu) = \widehat{\beta}_n$. Also, $\mu_n(\mathfrak{m}, M^\nu) = \beta_n$ (see Section 3 for definitions).

\section{Preliminaries}

We recall that a module $G$ is Gorenstein projective if there is an exact and $Hom
(-, Proj)$ exact complex $ \ldots \rightarrow P_1 \rightarrow P_0
\rightarrow P_{-1} \rightarrow P_{-2}\rightarrow \ldots $ of
projective modules such that $G = Ker(P_0 \rightarrow P_{-1})$.

\begin{definition}
 A module $_RM$ has finite Gorenstein projective dimension if there exists an exact sequence $0 \rightarrow G_n \rightarrow G_{n-1} \rightarrow \ldots \rightarrow G_1 \rightarrow G_0 \rightarrow M \rightarrow 0$ with all $G_i$ Gorenstein projective modules. If the integer $n \ge 0$ is the least with this property then $M$ has Gorenstein projective dimension $n$ (in short, $G.p.d._R(M)=n$). If no such $n$ exists then $M$ has infinite Gorenstein projective dimension.\\ The Gorenstein injective modules, and Gorenstein injective dimension are defined dually.
\end{definition}
The Tate cohomology modules are defined by means of complete resolutions. We recall the definition:\\

\begin{definition}
A module $M$ has a complete projective resolution if there exists a diagram $T \xrightarrow{u} P \rightarrow M $ with $P \rightarrow M$ a projective resolution, $T$ a totally acyclic complex and $u:T \rightarrow P$ a map of complexes such that $u_n:T_n \rightarrow P_n$ is an isomorphism for all $n \gg 0$.
\end{definition}

It is known that a module $M$ has a complete projective resolution if and only if and only if it has finite Gorenstein projective dimension. In particular, when $R$ is Gorenstein, every $_RM$ has a complete projective resolution.\\
Complete injective resolutions are defined dually. It is known that a module $_RN$ has a complete injective resolution if and only if $N$ has finite Gorenstein injective dimension. Over a Gorenstein ring $R$, every module $N$ has such a complete injective resolution.

\section{Tate-Betti numbers and Tate-Bass numbers}

Let $R$ be a commutative local noetherian ring, 
 and let $M$ be an $R$-module of finite Gorenstein projective dimension.  Then there is a complete projective resolution of $M$, $T \rightarrow P \rightarrow M$.
 If $M$ is finitely generated
 then we can choose $P$ to be a minimal projective resolution of $M$ (\cite{enochs:00:relative}).\\
 We recall (\cite{enochs:00:relative2}) that a complex $C$ is said to be homologically minimal if any homology isomorphism $f:C \rightarrow C$ is an isomorphism in $C(R-Mod)$.
 A complex $C$ is said to be homotopically minimal if each homotopy isomorphism $f:C \rightarrow C$ is an isomorphism. So if $C$ is homologically minimal, it is also homotopically minimal. \\
 Thus a minimal projective resolution $P $ of $M$, as above,
is homotopically minimal and in fact homologically minimal (see page 78 in
chapter 8 of \cite{enochs:00:relative2}), and so such a $P$ is unique
up to isomorphism.\\
We show first of all that when $M$ is finitely generated we can
also get $T$ to be homotopically minimal, and so also unique up to
isomorphism.\\

We will use the following\\

\begin{lemma} Let $K$ be a finitely generated Gorenstein projective reduced $R$-module (i.e. $K$ has no nontrivial projective direct summands). Then there exists an exact and $Hom(-,Proj)$ exact complex $0 \rightarrow K \rightarrow Q_0 \rightarrow Q_{-1} \rightarrow Q_{-2} \rightarrow \ldots$ with each $Q_n$ a finitely generated free module.\\
\end{lemma}

\begin{proof}
Let $K$ be finitely generated, Gorenstein projective and reduced.
Then the dual $K^*=Hom(K,R)$ is also such. Also if $0\rightarrow K'
\rightarrow P\rightarrow K\rightarrow 0$ is exact where $P\rightarrow K$
is a projective cover of $K$, then $K'$ is also finitely generated, Gorenstein
projective and reduced. \\
Since the dual module, $K^*$, is a finitely generated Gorenstein projective module that is also reduced there exists a short exact sequence $0 \rightarrow L \rightarrow P \rightarrow K^* \rightarrow 0$ with $P \rightarrow K^*$ a projective cover, and with $L$ Gorenstein projective finitely generated and reduced. This gives an exact sequence $0 \rightarrow K^{**} \rightarrow P^* \rightarrow L^* \rightarrow 0$ with $P^*$ finitely generated projective, and with $L^*$ finitely generated Gorenstein projective and reduced. Then $K^{**} \rightarrow P^*$ is a projective preenvelope, and therefore if $K^{**} \rightarrow Q$ is the projective envelope then $K^{**} \rightarrow Q$ is an injective map; also, since $coker(K\rightarrow Q)$ is a direct summand of $L^*$, it is finitely generated Gorenstein projective and reduced.
So $K$ has a projective envelope $K\rightarrow Q$
where $K\rightarrow Q$ is an injection and where $coker(K\rightarrow Q)$
is also finitely generated, Gorenstein projective and reduced.
Also, if $0\rightarrow K\rightarrow Q \rightarrow K_0 \rightarrow
0$ is exact, then $Q\rightarrow K_0$ is a projective cover (\cite{enochs:00:relative}, Proposition 10.2.10).\\
 So there exists a short exact sequence $0 \rightarrow K \rightarrow Q \rightarrow K_0 \rightarrow 0$ with $K \rightarrow Q$ a projective preenvelope and with $K_0$ a finitely generated reduced Gorenstein projective $R$-module.\\
Continuing, we obtain an exact and $Hom(-,Proj)$-exact complex $0 \rightarrow K \rightarrow Q \rightarrow Q_{-1} \rightarrow Q_{-2} \rightarrow \ldots$ with each $Q_n$ a finitely generated free $R$-module.
\end{proof}

We can prove now that when $M$ is a finitely generated $R$-module of finite Gorenstein projective dimension we can construct a complete projective resolution $T \rightarrow P \rightarrow M$ with both $T$ and $P$ unique up to isomorphism. \\
To see this let
\begin{center}
{$ 0\rightarrow K_m\rightarrow P_{m-1} \rightarrow \cdots \rightarrow
P_0 \rightarrow 0$}
\end{center}
be a partial minimal projective resolution of $M$ but where $m=G.p.d._RM$. 
Then 
$K_m$ is Gorenstein projective
and  since the resolution is minimal, $K_m$ is also reduced (i.e. has no nontrivial projective direct summands).\\
Using Lemma 1 above it is not hard to see that there exists an exact and $Hom(-, Proj)$ exact complex $T$ with each $T_n$ finitely generated free module, and with $K_m= Ker(T_{m-1} \rightarrow T_{m-2})$. Since for each d, the complex $\ldots \rightarrow T_{d+2} \rightarrow T_{d+1} \rightarrow K_d \rightarrow 0$ is a minimal projective resolution, it follows that $T$ is a homologically minimal complex (\cite{enochs:00:relative2}, page 78).\\

So for a finitely generated $M$
we can construct a complete projective resolution
\begin{center}
{$T\rightarrow P\rightarrow M$}
\end{center}
 where both $T$ and $P$ are homotopically minimal, and so unique up to
isomorphism. We call such a diagram a \emph{minimal complete projective resolution}
of $M$.\\
Then each of $P_n$ ($n\geq 0$) and $T_n$ ($n$ arbitrary) are free modules of
finite rank. As usual, the ranks of the $P_n$ are denoted $\beta_n(M)$. We denote
the ranks of the $T_n$ by $\widehat{\beta}_n (M)$. The numbers $\beta_n(M)$ are called
the Betti invariants of $M$.
We call the invariants $\widehat{\beta}_n(M)$ the Tate-Betti
invariants of $M$ (see \cite{christensen:14:tate} for another way
to define these invariants).\\

For an arbitrary module $N$ we can use an analogous procedure to
construct a complete injective resolution
\begin{center}
{$ N\rightarrow I \rightarrow U$}
\end{center}
where both $I$ and $U$ are homotopically minimal complexes (and hence
unique up to isomorphism). Then using Matlis and Bass, we can define
the Bass invariants $\mu^n ({\mathfrak p}, N) $ for $n\geq 0$ and ${\mathfrak
p}\subset R$ a prime ideal of $R$,  and then the Tate-Bass invariants
$\widehat{\mu}^n ({\mathfrak p}, N)$ for arbitrary $n$ and $\mathfrak p$ a
prime ideal.\\

Our main results are about the periodicity of these invariants. We recall first the following:\\
\begin{definition} A complex $C=(C_n)$ is said to be eventually periodic
of period $s\geq 1$ if for some $n_0$ we have that for all $n\geq n_0$
that
\begin{center}
{$ (C_{n+1}\rightarrow C_n \rightarrow C_{n-1}) \cong (C_{n+s+1} \rightarrow
C_{n+s}\rightarrow C_{n+s-1})$}
\end{center}
\end{definition}

Saying that $C$ is periodic of period $s$ will have
the obvious meaning.\\

\begin{remark} If $T\rightarrow P\rightarrow M$ is a minimal complete
projective resolution of a finitely generated $M$ and if $P$ is
eventually periodic, then trivially $T$ is also eventually periodic. But
using the minimality of $T$ it can be seen that in fact $T$ is periodic.
If this is the case where the period is $s$ then we see that
$\widehat{\beta}_n (M)=\widehat{\beta}_{n+s}(M) $ for all $n$. So we can say
that the Tate-Betti invariants are periodic of period $s$.\\
However it may happen that the Tate-Betti invariants of $M$ are periodic
without $T$ being periodic. So we can speak of the invariants being
periodic without the associated complex being periodic.\\
\end{remark}

We prove that when the residue field $k$ has an eventually periodic minimal projective resolution the Tate-Betti numbers of any finitely generated $_RM$ of finite Gorenstein projective dimension are periodic. Then we prove that under the same hypothesis on $k$, the Tate-Bass numbers $\widehat{\mu}^n (\mathfrak{m}, N)$ are periodic, for any $_RN$ of finite Gorenstein injective dimension.

We will use \cite{enochs:12:balance.with.unbounded} to deduce the following balance result
(see also \cite{christensen:14:tate}, Section 5).\\

\begin{theorem}\label{aside} Let $R$ be any commutative ring. If $T\rightarrow P\rightarrow M$ and $N \rightarrow I \rightarrow U$  are complete
projective and injective resolutions of $M$ and $N$ respectively (equivalently $G.p.d._R(M)$ and $G.i.d._R(N)$ are finite). Then
the homologies of $Hom(T,N)$ and of $Hom(M, U)$  are naturally isomorphic.\\
The analogous result for $T\otimes N$ and $M\otimes U$ also holds.
\end{theorem}
\begin{proof}
Without lost of generality let us assume $G.p.d._R(M)=m$ and $G.i.d._R(N)=n$ and $m\geq n$.
$$H^i(Hom(T,N))=H^{i-m-1}(Hom(T,C_{-m}))\cong $$ $$\cong H^{i-m-1}(Hom(K_m,U))= H^i(Hom(M,U)),$$
where $C_{-m}=Ker(U_{-m}\to U_{-m-1})$, $K_m=Ker(T_m\to T_{m-1}$ and $\cong$ follows from  (\cite{enochs:12:balance.with.unbounded}, Corollary 3.4). Note that if case $m>n$, we are using the fact that the class of Gorenstein injective modules is closed under cokernels of monomorphisms. The second statement follows in the same way.
\end{proof}

We can prove now our main results.\\
\begin{theorem} If the residue field $k$ of $R$ as an $R$-module has an eventually periodic
minimal projective resolution (with the period being $s$) then
for every finitely generated module $M$ of finite Gorenstein projective dimension we have that the Tate-Betti
invariants of $M$ are periodic of period $s$.\\
\end{theorem}

\begin{proof} We let $T\rightarrow P\rightarrow k$ be a minimal complete
projective resolution of $k$. Since $P$ is eventually periodic of period
$s$ we have that the complex $T$ is periodic of period $s$. Consequently
the complex $T\otimes M$ is periodic of period $s$. This gives
that for every $n$  $H_n(T\otimes M)\cong H_{n+s} (T\otimes M)$. But
by \cite{christensen:14:tate}, $H_n(T\otimes M)\cong H_n (k\otimes T')$ for
all $n$, where $T' \rightarrow P' \rightarrow M$ is a minimal complete projective resolution of $M$.\\ Since $T' = \ldots \rightarrow T'_{n+1} \xrightarrow{t'_{n+1}} T'_n \xrightarrow{t'_n} T'_{n-1} \rightarrow \ldots$ is a homologically minimal complex, it follows (Proposition 8.1.3 of \cite{enochs:00:relative2}) that $Im(t'_n) \subset \mathfrak{m} T'_{n-1}$ for all $n$. \\
Consider the complex $k\otimes T' = \ldots \rightarrow k \otimes T'_{n+1} \xrightarrow{\alpha_{n+1}} k \otimes T'_n \xrightarrow{\alpha_n} k \otimes T'_{n-1} \ldots $. We have $\alpha_n((x+\mathfrak{m}) \otimes y) = (x+\mathfrak{m}) \otimes t'_n(y)$; by the above $t'_n(y) = rz$ with $r \in \mathfrak{m}$ and $z \in T'_{n-1}$. So $\alpha_n((x+\mathfrak{m}) \otimes y) = (x+\mathfrak{m}) \otimes rz = (xr+\mathfrak{m}) \otimes z = (0 +\mathfrak{m}) \otimes z = 0$. Thus $Im(\alpha_n) = 0$ and $Ker(\alpha_n) = k \otimes T'_n$ for all $n$. Then the $n$-th homology module of $k\otimes T'$, $H_n (k \otimes T') = \frac{Ker(\alpha_n)}{Im(\alpha_{n+1})} = k \otimes T'_n \simeq k \otimes R^{\widehat{\beta}_n} \simeq k^ {\widehat{\beta}_n}$, is a vector space of dimension $\widehat{\beta}_n$ over $k$.

 Since $H_n(k\otimes T')\cong
H_n (T\otimes M)$ and since $H_n (T\otimes M)\cong H_{n+s} (T\otimes M)$
for all $n$ we see that $\widehat{\beta}_n(M)=\widehat{\beta}_{n+s}(M) $ for all $n$.
\end{proof}

\begin{theorem} If the residue field $k$ of $R$ as an $R$-module has an eventually periodic minimal projective resolution (with period $s\geq 1$), then for
any module $N$ of finite Gorenstein injective dimension, the invariants $\widehat{\mu}^n ( {\mathfrak m}, N)$ are periodic
of period $s$.
\end{theorem}

\begin{proof} Again let $T\rightarrow P\rightarrow k$ be a minimal
complete projective resolution of $k$ and let $N\rightarrow I\rightarrow
U$ be a minimal complete injective resolution of $N$. We have that
$Hom(T,N)$ and $Hom(k,U)$ have naturally isomorphic homology modules
(Theorem \ref{aside}). But $T$ is periodic of period $s$, so $Hom(T,N)$
is periodic of period $s$. So we get that
\begin{center}
{$ H^n(Hom(k, U)) \cong H^{n+s} (Hom(k, U))$}
\end{center}
for all $n$. But as in Bass (\cite{bass:63:ubiquity}) we see
that $H^n(Hom(k, U))$ is a vector space over $k$ whose dimension is
precisely $\widehat{\mu}^n ({\mathfrak m}, N)$.
\end{proof}

\begin{remark}
1. 
If M is eventually periodic, then its Betti sequence is bounded. The converse is not true in
general. A counterexample was given by R. Schulz in \cite{schultz:86:boundedness}, Proposition 4.1.  D. Eisenbud proved (\cite{eisenbud:80:complete.intersections}) that the converse does hold over group rings of finite
groups, and that it also holds in the commutative Noetherian local setting when
the rings considered are complete intersections. In fact, it was shown that over
a hypersurface (that is, a complete intersection ring of codimension one) any minimal
free resolution eventually becomes periodic.\\
So over a hypersurface ring both Theorem 2 and Theorem 3 hold.
\end{remark}

\begin{remark} Our main results hold provided that $k$ has an eventually periodic minimal projective resolution, and so its Betti numbers are bounded. By \cite{gulliksen:69:minimal} (Corollary 1), if the Betti numbers of $k$ are bounded then $R$ is a hypersurface.\\
So Theorems 2 and 3 both hold if and only if $R$ is a hypersurface.
\end{remark}

\section{Matlis duality}

Let $(R, \mathfrak{m}, k)$ be a commutative local Gorenstein ring, and let $M$ be a finitely generated $R$-module. Then there exists a diagram $T \rightarrow P \rightarrow M$ as above, with both $T$ and $P$ homotopically minimal, and so unique up to isomorphism.\\
Then for each $n \in \mathbb{Z}$, we have $P_n = R^{\beta_n}$ and $T_n = R^ {\widehat{\beta}_n}$.\\
We show that $M^\nu \rightarrow P^\nu \rightarrow T^\nu$ is a minimal complete injective resolution of the module $M^\nu$, where $M^\nu$ denotes the Matlis dual $Hom_R(M,E(k))$.\\
- Since $E(k)$ is injective, both $P^\nu$ and $T^\nu$ are exact complexes of injective modules.\\
Let $M_j = Ker(P_{j-1} \rightarrow P_{j-2})$; then $M_j ^ \nu \subset P_j ^\nu$ is an injective preenvelope with $P_j ^ \nu = Hom (R^{\beta_j}, E(k)) \simeq E(k)^{\beta_j}$. So the injective envelope of $M_j^\nu$ is a direct summand of $E(k)^{\beta_j}$, so it is $E(k)^{t_j}$ for some $1 \le t_j \le \beta_j$.\\
Then as in the proof of \cite{enochs:00:relative}, Corollary 3.4.4, there is an exact sequence $R^{t_j} \rightarrow M_j \rightarrow 0$. Therefore $R^{t_j} \rightarrow M_j$ is a projective precover. Since $P_j \rightarrow M_j$ is a projective cover, it follows that $R^{\beta_j}$ is a direct summand of $R^{t_j}$. So $\beta_j \le t_j$, and so we have $\beta_j = t_j$ for all $j$, and $M_j^\nu \rightarrow P_j^\nu $ is an injective envelope, for all $j$.\\

Since $T = \ldots \rightarrow T_1 \xrightarrow{t_1} T_0 \xrightarrow{t_0} T_{-1} \rightarrow \ldots$ is exact with each $T_j$ finitely generated free and with each $G_{j+1} = Ker(t_j)$ Gorenstein projective, it follows that $T^\nu$ is an exact complex of injective modules, with $Ker(T_j^\nu \rightarrow T_{j+1}^\nu) = G_j^\nu$ Gorenstein injective (we have that each $G_j$ is also Gorenstein flat in this case, so $Tor_1(G_j, A) =0$, for any injective module $A$, for each $j$. Then $Ext^1 (A, G_j^\nu) = Ext^1(A, Hom(G_j , E(k)) \simeq Hom(Tor_1(G_j, A) , E(k)) =0$, for any injective module $A$. It follows that $G_j^\nu$ is Gorenstein injective). 

Thus $T^\nu$ is a totally acyclic injective complex. As above, since $\ldots \to T_1 \rightarrow T_0 \rightarrow G_0 \rightarrow 0$ is a minimal projective resolution, it follows that $0 \rightarrow G_0^\nu \rightarrow T_0 ^\nu \rightarrow T_1 ^ \nu \rightarrow \ldots$ is a minimal injective resolution. \\

Similarly, $0 \rightarrow G_{-n-1}^\nu \rightarrow T_{-n-1}^\nu \rightarrow T_{-n}^\nu \rightarrow \ldots$ is a minimal injective resolution of $G_{-n-1}^\nu$. By \cite{enochs:00:relative}, Theorem 10.1.13, the Gorenstein injective module $G_{-n-1}^\nu$ is reduced. Thus $T_{-n-1}^\nu \rightarrow G_{-n}^\nu$ is in fact an injective cover. Similarly we have that $T_j^\nu \rightarrow G_{j+1}^\nu$ is an injective cover for each $j \le 0$.\\

Thus $ \ldots T_{-2}^\nu \rightarrow T_{-1}^\nu \rightarrow G_0^\nu \rightarrow 0$ is a minimal left injective resolution.


So we have $M^\nu \rightarrow P^\nu \rightarrow T^\nu$ a complete injective resolution of $M^\nu$ with both $P^\nu$ and $T^\nu$ minimal.\\
We have $P_n^\nu = E(k)^{\beta_n}$, and $T_n^\nu = E(k)^{\widehat{\beta}_n}$, for each $n$. It follows that for each $n$ we have $\widehat{\mu}_n(\mathfrak{m}, M^\nu) = \widehat{\beta}_n$. Also, $\mu_n(\mathfrak{m}, M^\nu) = \beta_n$ for each $n \ge 0$.

\bibliographystyle{plain}

\end{document}